\documentclass{birkmult}
 \newtheorem{thm}{Theorem}[section]
 \newtheorem{cor}[thm]{Corollary}
 \newtheorem{lem}[thm]{Lemma}
 \newtheorem{prop}[thm]{Proposition}
 \theoremstyle{definition}
 \newtheorem{defn}[thm]{Definition}
 \theoremstyle{remark}
 \newtheorem{rem}[thm]{Remark}
 
 \numberwithin{equation}{section}

\begin{document}
%
\title[Non linear quasi-elliptic equations]
 {Microlocal regularity of Besov type
 \linebreak
for solutions to quasi-elliptic non linear\linebreak
 partial differential equations}
\author[G. Garello]{Gianluca Garello}

\address{%
Department of Mathematics\\
University of Torino\\
Via Carlo Alberto 10
\\
I- 10123 Torino, Italy}

\email{gianluca.garello@unito.it}
\author[A. Morando]{Alessandro Morando}
\address{DICATAM
\\ University of Brescia\\
Via Valotti 9\\
I-25133 Brescia,
Italy}
\email{ alessandro.morando@unibs.it}


\begin{abstract}
Using a standard linearization technique and previously obtained microlocal properties for pseudodifferential operators with smooth coefficients, the authors state results of microlocal regularity in generalized Besov spaces for solutions to non linear PDE.
\end{abstract}

\maketitle


\section{Introduction}\label{INT}
In previous papers, \cite{GM1}, \cite{GM2}, \cite{GM3}, \cite{GM3A}, the authors studied the problem of $L^p$ and Besov continuity and local regularity for pseudodifferential operators with smooth and non smooth symbols, whose derivatives decay at infinity in non homogeneous way. Particularly in  \cite{GM3}, \cite{GM3A} emphasis is given on symbols with quasi-homogeneous decay; in \cite{GM4} also microlocal properties were  studied.\\
Pseudodifferential operators whose smooth symbols have a quasi-homo\-ge\-neo\-us decay at infinity were first introduced in 1977 in Lascar \cite{LA1}, where their microlocal properties in the $L^2-$framework were studied.\\
Symbol classes of quasi-homogeneous type and several related problems have been developed in the meantime, see e.g. Seg\`ala \cite{SE1} for the local solvability, Garello \cite{GA1} for symbols with decay of type $(1,1)$, Yamazaki \cite{YA1} where  non smooth symbols in the $L^p-$framework are introduced and studied under suitable restrictive conditions on the Fourier transform of the symbols themselves.\\
The aim of the present paper is to apply the previous results to the study of microlocal properties  of fully non linear equations, by means of the linearization techniques introduced by M. Beals and M.C. Reeds in \cite{BR1} and well described in \cite{Ta1}, \cite{G1}.  Namely,  consider the non linear equation
\begin{equation}\label{eqINT1}
F(x,\partial^\alpha u)_{\alpha\in \mathcal I}=0,
\end{equation}
where $F(x,\zeta)\in C^\infty(\mathbb R^n\times \mathbb C^N)$ for suitable positive integer $N$, and  $\mathcal I$ is a bounded subset of multi-indices in $\mathbb Z_+^n$. After the linearization obtained by differentiating with respect to the $x_j$ variable:
\begin{equation}\label{eqINT2}
\sum_{\alpha\in\mathcal I}\frac{\partial F}{\partial \zeta^\alpha}(x, \partial^\beta u)_{_{\beta \in\mathcal I}}\partial^\alpha\partial_{x_j}u=-\frac{\partial F}{\partial_{x_j}}(x,\partial^\beta u)_{_{\beta \in\mathcal I}},
\end{equation}
we reduce the study of \eqref{eqINT1} to the following linear equation
\begin{equation}\label{eqINT3}
\sum_{\alpha\in\mathcal I}a_\alpha(x)\partial^\alpha u_j=f_j(x), \quad u_j=\partial_{x_j}u,
\end{equation}
where the coefficients $a_\alpha(x)$ and the forcing term $f_j(x)$ are clearly non smooth, but their regularity depends on $u$ itself. Precisely here we are considering the regularity of solutions to \eqref {eqINT1} in the framework of quasi-homogeneous Besov spaces $B^{s,M}_{\infty, \infty}$, which are introduced in Section \ref{BS}, by means of a suitable decomposition of $\mathbb R^n$ in anisotropic dyadic crowns.\\
In \S\ref{qomsymbols}  pseudodifferential operators and symbol classes are defined and in \S\ref{mcl} we introduce  the microlocal properties of Besov type $B^{s,M}_{\infty,\infty}$ for pseudodifferential operators with smooth symbols, obtained in \cite{GM4}. Such results apply in \S \ref{appl} to the study of the microlocal regularity for solution to equations of type \eqref{eqINT3}, with coefficients of Besov type and, in the last section, to quasi-linear and fully non linear equations.

\section{Quasi-homogeneous Besov spaces}\label{BS}
In the following  $ M=(m_1,\dots, m_n)$ is a weight vector with positive integer components, such that $\min\limits_{1\le j\le n}m_j=1$ and
\begin{equation}\label{QOWF}
\vert\xi\vert_M:=\left( \sum_{j=1}^{n}\xi_j ^{2m_j}\right)^{\frac12}, \quad \xi\in \mathbb R^n
\end{equation}
is called {\it quasi-homogeneous weight function on} $\mathbb R^n$.

We set  $m^\ast:=\max\limits_{1\le j\le n}m_j$, $\frac 1M:=\left( \frac{1}{m_1},\dots, \frac{1}{m_n}\right)$, $\alpha\cdot\frac1M=\sum_{j=1}^n \frac{\alpha_j}{m_j}$ and $\langle \xi\rangle^2_M:= (1+\vert \xi\vert_M ^2)$. Clearly the usual euclidean norm $\vert\xi\vert$ corresponds to the quasi-homogeneous weight in the case $M=(1,\dots,1)$.
\newline
By easy computations, see e.g. \cite{GM3} we obtain the following
\begin{prop}\label{WPROP}
For any weight vector $M$ there exists  a suitable positive constant $C$ such that
\begin{itemize}
\item[i)]
$\frac 1 C\langle \xi\rangle\le \langle\xi\rangle_M\le C\langle\xi\rangle^{m^\ast}, \quad \xi\in\mathbb R^n$,
\item[]
\item[ii)]
$\vert\xi+\eta\vert_M\leq C(\vert \xi\vert_M+\vert \eta \vert_M), \quad \xi,\eta\in\mathbb R^n$;
\item[]
\item [iii)](quasi-homogeneity) for any $t>0$, $\vert t^{ 1/M}\xi\vert_M =t\vert\xi\vert_M$,
\newline
where $t^{1/M}\xi=(t^{1/m_1}\xi_1,\dots,t^{1/m_n}\xi_n)$;
\item[]
\item[iv)]
$\xi^\gamma\partial^{\alpha+\gamma}\vert\xi\vert_M\leq C_{\alpha,\gamma} \langle\xi\rangle_M^{1-\alpha\cdot \frac1M}$,\,\, for any\,\, $\alpha,\gamma\in\mathbb{Z}_+^n\,and\,\,\xi\neq 0$.
\end{itemize}
\end{prop}
For $t>0$, $h\geq-1$ integer, we introduce the  notations: $t^{\frac h{\vert M\vert}}=t^\frac h{m_1}\dots t^\frac h{m_n}$ and $t^{\frac hM}\xi=\left(t^{\frac{h}{m_1}}\xi_1, \dots, t^{\frac{h}{m_n}}\xi_n\right)$.\\
In the following $\hat u(\xi)=\mathcal F u(\xi)=\int e^{- i x\cdot\xi}u(x)\, dx\,$ stands for both the Fourier transform of $u\in\mathcal S(\mathbb R^n)$ and its extension to $\mathcal S'(\mathbb R^n)$.
\begin{prop}\label{QBI}
Consider $u\in L^\infty(\mathbb R^n)$, $R>0$,  such that $\textup{supp}\,\hat u\subset B^M_R:=\left\{ \xi\in\mathbb R^n\, ;\, \vert \xi\vert_M\leq R\right\}$. Then for any $\alpha\in\mathbb Z^n_+$ there exists $c_\alpha>0$ independent of $R$ such that
\begin{equation}\label{eqQBI1}
\Vert \partial^\alpha u\Vert_{L^\infty}\leq c_\alpha R^{\alpha\cdot \frac1M}\Vert u\Vert_{L^\infty}.
\end{equation}
\end{prop}
\begin{proof}
Consider $\phi\in C^\infty(\mathbb R^n)$ such that $\textup{supp}\, \phi \subset  B^M_2$, $\phi(x)=1$ in $B^M_1$ and set $\phi_R(\xi)= \phi\left( R^{-\frac1M}\xi\right)$. Since $\phi_R(\xi)=1$ in $B^M_R$, we  obtain $\hat u(\xi)=\phi_R(\xi)\hat u(\xi)$. Thus
\begin{equation}\label{eqQBI2}
\begin{array}{ll}
\displaystyle u&=\mathcal F^{-1}\left( \phi_R\hat u\right)=\mathcal F^{-1}\phi_R\ast u=\\
&\\
& \displaystyle=(2\pi)^{-n}R^{\frac1{|M|}}\left(\int e^{i \left(R^{\frac1{M}}\cdot\right)\cdot \eta}\phi(\eta)\,d\eta \ast u\right)\\
&\\
& \displaystyle=(2\pi)^{-n}R^{\frac 1{\vert M\vert}}\left(\hat \phi\left(-R^{\frac 1M}\cdot\right)\ast u\right)\in C^\infty(\mathbb R^n)\,,
\end{array}
\end{equation}
where $\mathcal F^{-1}$ denotes the inverse Fourier transform. Moreover
\begin{equation}\label{eqQBI3}
\partial ^\alpha u=(2\pi)^{-n}(-1)^{\vert \alpha\vert}R^{\frac {1}{\vert M\vert}}R^{\alpha\cdot \frac 1M}(\partial^\alpha \hat \phi)\left(-R^{\frac1M}\cdot\right)\ast u\,.
\end{equation}
Then
\begin{equation}\label{eqQBI4}
\begin{array}{ll}
\displaystyle\Vert \partial^\alpha u\Vert_{L^\infty}&\leq(2\pi)^{-n}R^{\frac{1}{\vert M\vert}}\Vert (\partial^\alpha\hat \phi)\left(-R^{\frac 1M}\cdot\right)\Vert_{L^1}\,R^{\alpha\cdot\frac 1M}\Vert u\Vert_{L^\infty}\\
\\
& \displaystyle = (2\pi)^{-n}R^{\frac{1}{\vert M\vert}}\int\left|(\partial^\alpha\hat\phi)\left(-R^{\frac 1M}\xi\right)\right|d\xi\, R^{\alpha\cdot\frac 1M}\Vert u\Vert_{L^\infty}\\
\\
& \displaystyle = (2\pi)^{-n}\int\left|\partial^\alpha\hat\phi(\eta)\right|d\eta\, R^{\alpha\cdot\frac1M}\Vert u\Vert_{L^\infty}=c_\alpha R^{\alpha\cdot\frac 1M}\Vert u\Vert _{L^\infty}.
\end{array}
\end{equation}
\end{proof}

\begin{prop}[Quasi-homogeneous dyadic decomposition]\label{DYAD}
For some $K>1$ let us consider  the cut-off function $\phi(t)\in C^\infty_0\left([0,+\infty]\right)$ such that $0\leq \phi(t)\leq 1$, $\phi(t)=1$ for $0\leq t\leq \frac{1}{2K}$, $\phi(t)=0$, when $t>K$. Set now $\varphi_0(\xi)=\phi\left(\left\vert 2^{-1/M}\xi\right\vert_M\right)-\phi(\vert\xi\vert_M)$ and
\begin{equation}\label{SC2BIS}
\varphi_{-1}(\xi) =\phi\left(\vert \xi\vert_M\right), \quad \varphi_h(\xi)=\varphi_0\left(2^{-h/M}\xi\right)\,\,\mbox{for}\,\,h=0,1,\dots\,.
\end{equation}
Then for any $\alpha,\gamma\in \mathbb Z^n_+$ a positive constant $C_{\alpha,\gamma,K}$ exists such that:
\begin{eqnarray}
&&\begin{array}{l}
\text{supp}\, \varphi_{-1}\subset C_{-1}^{K,M}:= B^M_K\\
\\
\text{supp}\, \varphi_h\subset C_h^{K,M}:=\left\{\xi\in\mathbb{R}^n\,\, ; \,\,\frac1K 2^{h-1}\leq\vert\xi\vert_M\leq K2^{h+1}\right\},\,\,h\geq	0;
\end{array}
\label{SC4BIS}\\
&&\sum\limits_{h=-1}\limits^\infty\varphi_h(\xi)=1\,,\,\,\text{for all}\,\,\xi\in\mathbb R^n; \label{SC5}\\
&&\left\vert \xi^\gamma\partial^{\alpha+\gamma}\varphi_h(\xi)\right\vert \leq  C_{\alpha,\gamma, K}2^{-\left(\alpha\cdot \frac1M\right)h},\quad \xi\in\mathbb R^n\,,\,\,h=-1,0,\dots\,.
\label{SC6}
\end{eqnarray}
Moreover for any fixed $\xi\in\mathbb{R}^n$ the sum in \eqref{SC5} reduces to a finite number of terms, independent of the choice of $\xi$ itself. \\
Setting now for every $u\in\mathcal S'(\mathbb R^n)$
\begin{equation}\label{UP}
u_h=\varphi_h(D)u:=\mathcal F^{-1} \left(\varphi_h\hat u\right),
\end{equation}
we obtain:
\begin{equation}\label{SC5BIS}
\sum\limits_{h=-1}\limits^\infty u_h=u, \quad\text{with convergence in }\,\,\mathcal S'(\mathbb R^n);
\end{equation}
and, for every integer $k\geq 0$ there exists $C_k>0$ such that 
\begin{equation}\label{SC5TER}
 \frac{1}{C_k}2^{hk}\Vert u_h\Vert_{L^\infty}\leq \sum_{\alpha\cdot\frac1{M}=k}\Vert \partial ^\alpha u_h\Vert_{L^\infty}\leq C_k 2^{hk}\Vert u_h\Vert_{L^\infty}, \quad  h=0,1\dots.
\end{equation}
\end{prop}
\begin{proof}
It is trivial to prove \eqref{SC4BIS}. For every fixed $\xi\in \mathbb{R}^n$ we have $\phi\left(\left\vert 2^{-h/M}\xi\right\vert_M \right)=1$ for any suitably large integer $h$; then \eqref{SC5}, \eqref{SC5BIS} follow.
\newline
For every integer $h\geq 0$ we obtain
\begin{equation*}\label{SC7}
\left\vert \xi^\gamma\partial ^{\alpha+\gamma}\varphi_h(\xi)\right\vert=\left\vert \left(2^{-h/M}\xi\right)^\gamma\left(\partial^{\alpha+\gamma}\varphi_0\right)\left( 2^{-h/M}\xi\right)\right\vert 2^{-h\left(\alpha\cdot\frac 1M\right)}\leq C_{\alpha,\gamma,K}2^{-h\left(\alpha\cdot\frac1M\right)},
\end{equation*}
where $C_{\alpha,\gamma,K} =\max\limits_\eta \left\vert \eta^\gamma\partial^{\alpha+\gamma}\varphi_0(\eta)\right\vert$ is independent of $h$; thus \eqref{SC6} is proved.\\
In order to prove at the end \eqref{SC5TER}, let us consider $\chi(\xi)\in C^\infty_0(\mathbb R^n)$ identically equal to one in a suitable neighborhood of supp $\varphi_0$. We can then write
\begin{equation}\label   {CRA1}
\varphi_0(\xi)=\left( \sum_{\alpha\cdot\frac 1M=k}\xi^\alpha \chi_\alpha(\xi)\right)\varphi_0(\xi),
\end{equation}
with
\begin{equation}\label{CRA2}
\chi_\alpha(\xi)= \frac{\xi^\alpha\chi(\xi)}{\sum_{\alpha\cdot\frac1M=k}(\xi^\alpha)^2}\in C^\infty_0(\mathbb R^n).
\end{equation}
Thus we have:
\begin{equation}\label{CRA3}
\begin{array}{ll}
\displaystyle\hat u_h(\xi)=\varphi_0\left(2^{-\frac hM}\xi\right)\hat u(\xi)& \displaystyle=\sum_{\alpha\cdot\frac1M =k}\left(2^{-\frac hM}\xi\right)^\alpha\chi_\alpha\left(2^{-\frac hM}\xi\right)\hat u(\xi)\\
& \displaystyle =\sum_{\alpha\cdot\frac1M =k}2^{-h\alpha\cdot\frac1M}\xi^\alpha \chi_\alpha \left(2^{-\frac hM}\xi\right)\hat u_h(\xi)\\
& \displaystyle =2^{-hk}\sum_{\alpha\cdot\frac1M =k}\chi_\alpha \left(2^{-\frac hM}\xi\right)\widehat{D^\alpha u_h}(\xi)\,.
\end{array}
\end{equation}
We have then verified.
\begin{equation}\label{CRA4}
2^{hk}u_h=\sum_{\alpha\cdot\frac1M =k}2^{\frac {h}{\vert M\vert}}\left(\left(\mathcal F^{-1}\chi_\alpha\right)(2^{\frac hM}\cdot)\right)\ast D^\alpha u_h,
\end{equation}
which in view of the Young inequality and Proposition \ref{QBI} shows \eqref{SC5TER}.
\end{proof}
We call the sequences $\varphi:=\{\varphi_h\}_{h=-1}^\infty$, defined in \eqref{SC2BIS}, and $\{u_h\}_{h=-1}^\infty$, defined in \eqref{UP}, respectively  {\it quasi-homogeneous partition of unity} and  {\it quasi-homogeneous dyadic decomposition} of $u$.\\
Following the arguments in \cite[\S 10.1]{TR} we can introduce now the classes of quasi-homogeneous  Besov functions and state their properties in suitable way.
\begin{defn}
For any $s\in\mathbb R$ and $u\in \mathcal S'(\mathbb R^n)$ we say that $u$ belongs to the {\it quasi-homogeneous Besov space} $B^{s,M}_{\infty,\infty}$ if
\begin{equation}\label{HOLDER}
\Vert u \Vert^\varphi_{B^{s,M}_{\infty,\infty}}:=\sup_{h=-1,\dots}2^{sh}\Vert u_h\Vert_{L^\infty}<\infty
\end{equation}
is satisfied for some quasi-homogeneous partition of unity $\varphi$.
\end{defn}
Different choices of the partition of unity $\varphi$ in \eqref{HOLDER} give raise to equivalent norms, noted by $\Vert \cdot\Vert_{B^{s,M}_{\infty,\infty}}$. The space $B^{s,M}_{\infty,\infty}$ has Banach structure and when $M=(1,\dots,1)$ and $s>0$, it is the usual  H\"older-Zygmund space.
\begin{prop}\label{CHAR}
Let us consider a sequence of Schwartz distributions $\{u_h\}_{h=-1}^\infty\subset\mathcal S'(\mathbb R^n)$ and a constant $K>1$ such that supp $\hat u_h\subset C_h^{K,M}$ for any $h\ge -1$. Set now $u:=\sum_{h=-1}^\infty u_h$.
\newline
The following properties are satisfied:
\begin{equation}\label{HOLDERCHAR}
\begin{array}{ll}
&\sup\limits_{h\ge -1}\left(2^{rh}\Vert u_h \Vert_{L^\infty}\right)<\infty\Rightarrow u\in B^{r,M}_{\infty,\infty}, \quad r\in \mathbb R,\\
\text{and}&\\
&\Vert u\Vert_{B^{r,M}_{\infty,\infty}}\le C \sup\limits_{h\ge -1}\left(2^{rh}\Vert u_h \Vert_{L^\infty}\right),
\end{array}
\end{equation}
where the constant $C$ is independent of the sequence $\{u_h\}_{h=-1}^\infty$.\\
When  $r>0$, \eqref{HOLDERCHAR} is true for all the sequences of Schwartz distributions $\{u_h\}_{h=-1}^{\infty}$ with supp $\hat u_h\subseteq B^{K,M}_h:=B^M_{K2^{h+1}}=\{\xi\in\mathbb{R}^n:\,\,\vert\xi\vert_M\leq K2^{h+1}\}$, $h=-1,0,\dots$.
\end{prop}
\begin{prop}[Quasi-homogeneous Meyer multipliers]\label{MEMU}
Consider a family of smooth functions $\{ m_h\}_{h=-1}^\infty$ such that for any $\alpha\in\mathbb Z_+^n$:
\begin{equation}\label{eqMEMU1}
\Vert\partial^\alpha m_h\Vert_{L^\infty}\leq C_\alpha 2^{h\, \alpha\cdot\frac1M}.
\end{equation}
Then the linear operator $L=\sum_{h=-1}^\infty m_h(x)\varphi_h(D)$ maps continuously $B^{s,M}_{\infty,\infty}$ into itself, for any $s>0$.
\end{prop}
\begin{proof}
Consider the quasi-homogeneous partition of unity in Proposition \ref{DYAD}, with $K=1$, and for any $h=-1,0,\dots$ and $T>2$ write:
\begin{equation}\label{eqMEMU2}
\hat m_h=\sum_{k=-1}^\infty\varphi_k\left(\left(2^hT\right)^{-\frac 1M}\cdot \right)\hat m_h=\sum_{k=-1}^\infty\hat m_{h,k}\,.
\end{equation}
Notice that $\hat m_{h,-1}(\xi)=\phi\left(\left|\left(2^h T\right)^{-\frac 1M}\xi\right|_M\right)\hat m_h(\xi)$, and  when $h\geq 0$, $\hat m_{h,k}(\xi)=\varphi_0\left(\left(2^{h+k}T\right)^{-\frac 1M}\xi\right)\hat m_h(\xi)$.\\
Thus for any $u\in\mathcal S'(\mathbb R^n)$, by setting $M_ku=\sum_{h=-1}^\infty m_{h,k} u_h$ for $k\ge -1$, we have:
\begin{equation}\label{eqMEMU3}
Lu=\sum_{h=-1}^\infty m_h\varphi_h(D)u=\sum_{h=-1}^\infty m_h u_h=\sum_{
k=-1\atop
h=-1
}^\infty m_{h,k}u_h=\sum_{k=-1}^\infty M_k u.
\end{equation}
Notice now that for any $h,k\geq -1$:
\begin{equation}\label{eqMEMU4}
\begin{array}{ll}
 \Vert m_{h,k}u_h\Vert_{L^\infty}\leq \Vert m_{h,k}\Vert_{L^\infty}\Vert u_k\Vert_{L^\infty};&\\
&\\
\textup{supp}\, \widehat{m_{h,-1}u_h}\subset B^{T,M}_h+C^{1,M}_h\subset B_h^{K,M};&\\
&\\
\textup{supp}\, \widehat{m_{h,k}u_h}\subset C^{T,M}_{h+k}+C^{1,M}_h\subset C_{h+k}^{K,M}, &\text{for suitable constants}\,\, T, K.
\end{array}
\end{equation}
Using now \eqref{SC5TER} and \eqref{eqMEMU1}, for any  integer $l>0$ there exist positive constants $C_l>0$ such that 
\begin{equation}\label{eqMEMU5}
\Vert m_{h,k}\Vert _{L^\infty}\leq C_l \sum_{\alpha\cdot \frac 1M=l}\Vert \partial^\alpha m_{h,k}\Vert_{L^\infty}2^{-(h+k)l}\leq C_l 2^{-kl}.
\end{equation}
Thus for any $s>0$:
\begin{equation}\label{eqMEMU6}
\begin{array}{ll}
2^{s(h+k)}\Vert m_{h,k}u_h\Vert_{L^\infty}&\leq 2^{s(h+k)}\Vert m_{h,k}\Vert_{L^\infty}\Vert u_h\Vert_{L^\infty}\\
\\
& \le C_l 2^{sh}2^{(s-l)k}\Vert u_h\Vert_{L^\infty}\leq C_l 2^{(s-l)k}\Vert u\Vert_{B^{s,M}_{\infty,\infty}}.
\end{array}
\end{equation}
Thus for any $s>0$, $l\geq 1$ and $k\geq -1$, in view of Proposition \ref{CHAR}, we get 
\begin{equation}\label{NUOVA}
\Vert M_k u\Vert_{B^{s,M}_{\infty, \infty}}\leq C_l 2^{(s-l)k}\Vert u\Vert_{B^{s,M}_{\infty,\infty}}.
\end{equation}
Then, by choosing $l>s$, in view of \eqref{eqMEMU3} and \eqref{NUOVA} we conclude that  $\Vert Lu\Vert_{B^{s,M}_{\infty,\infty}}\leq C_s\Vert u\Vert_{B^{s,M}_{\infty, \infty}}$.
\end{proof}
\begin{thm}\label{COMP}
Consider $F\in C^\infty(\mathbb C)$ such that $F(0)=0$, $s>0$. Then, for any $u\in B^{s,M}_{\infty,\infty}$ and suitable  $C=C(F, \Vert u\Vert_{L^\infty})$,  we have:
\begin{equation}\label{eqCOMP1}
F(u)\in B^{s,M}_{\infty,\infty} \quad\text{and}\quad \Vert F(u)\Vert_{B^{s,M}_{\infty,\infty}}\leq C\Vert u\Vert_{B^{s,M}_{\infty,\infty}}.
\end{equation}
\end{thm}
\begin{proof}
Using the notations in Proposition \ref{DYAD} let us define for any integer $p\geq 0$, $\Psi_p u=\Psi_p(D)u$, where $\Psi_p(\xi)=\sum\limits_{-1\leq h\leq p-1}\varphi_h(\xi)$. Since $F(0)=0$ setting moreover $\Psi_{-1}(\xi)=0$ we can consider the telescopic expansion:
\begin{equation}\label{eqCOMP2}
F(\Psi_0 u)+\sum_{p=0}^\infty \left(F(\Psi_{p+1}u)-F(\Psi_p u)\right)=\sum_{p=-1}^\infty \left(F(\Psi_{p+1}u)-F(\Psi_p u)\right)\,.
\end{equation}
By means of standard computations we have for any $p\geq 0$
\begin{equation}\label{eqCOMP3}
F(\Psi_{p+1}u)- F(\Psi_p u)=u_p\int_0^1 F'(\Psi_{p}u+t u_p)\,dt .
\end{equation}
Thus by setting
\begin{equation}\label{eqCOMP4}
 m_p(x)=\int_0^1  F'(\Psi_{p}u +t u_p)\, dt\,,
\end{equation}
we obtain $ F(u)=\sum_{p=-1}^\infty m_p u_p=Lu$. It is now sufficient to verify that $m_p$ defined in \eqref{eqCOMP4} is a Meyer multiplier.\\
Without any loss of generality, it is enough to consider $\tilde m_p=G\left(\Psi_p u\right)$, with $G= F'\in C^\infty$. Then
\begin{displaymath}
\partial^\alpha G\left(\Psi_pu\right)=\sum G^{(q)}\left(\Psi_pu\right)\left(\partial^{\gamma_1}\Psi_pu\right)\dots\left(\partial^{\gamma_q}\Psi_pu\right),\label{PP30}
\end{displaymath}
where $1\leq q\leq\vert\alpha\vert$, $\gamma_1+\dots+\gamma_q=\alpha$ and $\vert\gamma_j\vert\geq 1$, $j=1,\dots,q$.\\
It follows from the Proposition \ref{QBI}  that for any multi-index $\gamma_j$:
\begin{displaymath}
\Vert\partial^\gamma\Psi_p u\Vert_{L^\infty}\leq C 2^{p(\gamma_j\cdot \frac1M)}\Vert \Psi_p u\Vert_{L^\infty}.\label{PP31}
\end{displaymath}
Then for a suitable positive constant $C$ depending on $\alpha, G$ and $\Vert u\Vert_{L^\infty}$:
\begin{displaymath}
\Vert\partial^\alpha G\left(\psi_pu\right)\Vert_{L^\infty}\leq C 2^{p\left(\gamma_1\cdot\frac1{M}+\dots\gamma_q\cdot \frac 1{M}\right)}\leq C2^{p\left(\alpha\cdot\frac1{M}\right)},\label{PP32}
\end{displaymath}
which ends the proof.
\end{proof}
\begin{rem}\label{REMCOMP}
Set $\tilde F(t)=F(t)-F(0)$, with $F\in C^\infty(\mathbb C)$. Since the constant functions belong to $B^{s,M}_{\infty,\infty}$, we obtain that for any $u\in B^{s,m}_{\infty,\infty}$, $F(u)$ fulfills \eqref{eqCOMP1}, for any $s>0$.
\end{rem}
\section{Quasi-homogeneous symbols}\label{qomsymbols}
In this section, we recall the definition of some symbol classes   which are well behaved on the quasi-homogeneous structure of the spaces $B^{s, M}_{\infty, \infty}$. Here we just collect some basic definitions and a few related results, referring the reader to \cite{GM3, GM4} for a more detailed analysis. Let $M=(m_1,\dots, m_n)$ be a vector with positive integer components obeying the assumptions of the previous \S 2.
\begin{defn}\label{def:1}
For $m\in\mathbb{R}$ and $\delta\in[0,1]$, $S^m_{M,\delta}$ will be
the class of functions $a(x,\xi)\in
C^{\infty}(\mathbb{R}^n\times\mathbb{R}^n)$ such that for all $\alpha,\beta\in\mathbb{Z}^n_+$ there exists $C_{\alpha,\beta}>0$ such that:
\begin{equation}\label{eq:1}
\vert \partial^{\beta}_x\partial^{\alpha}_{\xi}a(x,\xi)\vert\le
C_{\alpha,\beta}\langle\xi\rangle_M^{m-\alpha\cdot 1/M+\delta\beta\cdot 1/M},\quad\forall x,\,\xi\in\mathbb{R}^n\,.
\end{equation}
We also set $S^m_M:=S^m_{M,0}$.
\end{defn}
\noindent
For each symbol $a\in S^m_{M,\delta}$, the pseudodifferential operator $a(x,D)={\rm Op}(a)$ is defined on $\mathcal{S}(\mathbb{R}^n)$ by the usual quantization
\begin{equation}\label{psdo}
a(x,D)u=(2\pi)^{-n}\int
e^{i\,x\cdot\xi}a(x,\xi)\hat u(\xi)\, d\xi,\quad u\in\mathcal S(\mathbb R^n).
\end{equation}
It is well-known that \eqref{psdo} defines a linear bounded operator from $\mathcal{S}(\mathbb{R}^n)$ to itself. In the following, we will denote by ${\rm Op}\,S^m_{M,\delta}$ the set of pseudodifferential operators with symbol in $S^m_{M,\delta}$ (and set ${\rm Op}\,S^m_{M}:={\rm Op}\,S^m_{M,0}$ according to Definition \ref{def:1}).
From Proposition \ref{WPROP}, iv), it is clear that $\langle\xi\rangle_M^m\in S^m_M$, for every $m\in\mathbb R$.\\
For pseudodifferential operators in ${\rm Op}\,S^m_{M,\delta}$, a suitable symbolic calculus is developed in \cite[Propositions 2.3-2.5]{GM4} under the restriction $\delta<\frac1{m^*}$; in particular the composition $a(x,D)b(x,D)$ of two operators $a(x,D)\in{\rm Op}\,S^m_{M,\delta}$, $b(x,D)\in{\rm Op}\,S^{m'}_{M,\delta}$ belongs to ${\rm Op}\,S^{m+m'}_{M,\delta}$, for all $m, m'\in\mathbb R$ as long as $\delta<\frac1{m^*}$.
\newline
The analysis of linear partial differential equations with rough coefficients needs the introduction of non smooth symbols studied in \cite{GM4}. We recall the  definitions and the main properties.
\begin{defn}\label{nonregularsymbols}
For $r>0$, $m\in\mathbb{R}$ and $\delta\in[0,1]$,
$B^{r,M}_{\infty,\infty}S^{m}_{M,\delta}$ is the set of  measurable functions $a(x,\xi)$ such that for every $\alpha\in\mathbb{Z}^n_+$
\begin{eqnarray}
\vert \partial^{\alpha}_{\xi}a(x,\xi)\vert\le
C_{\alpha}\langle\xi\rangle_M^{m-\alpha\cdot 1/M},\quad\forall
x,\,\xi\in\mathbb{R}^n;\label{eq:2}\\
\Vert \partial^{\alpha}_{\xi}a(\cdot,\xi)\Vert_{B^{r,M}_{\infty,\infty}}\le
C_{\alpha}\langle\xi\rangle_M^{m-\alpha\cdot 1/M+\delta
r},\quad\forall\xi\in\mathbb{R}^n.\label{eq:3}
\end{eqnarray}
\end{defn}
\noindent
As in  the case of smooth symbols, we set for brevity $B^{r,M}_{\infty,\infty}S^m_M:=B^{r,M}_{\infty,\infty}S^m_{M,0}$.
\begin{thm}\label{sobolev-holder-cont}
If $r>0$, $m\in\mathbb{R}$, $\delta\in[0,1]$ and $a(x,\xi)\in
B^{r,M}_{\infty,\infty}S^{m}_{M,\delta}$, then for all $s\in](\delta-1)r,r[$
\begin{eqnarray}
a(x,D):B^{s+m,M}_{\infty,\infty}\rightarrow B^{s,M}_{\infty,\infty}\label{boundedness2}
\end{eqnarray}
is a linear continuous operator.
\newline
If in addition $\delta< 1$, then the mapping property \eqref{boundedness2} is still true for $s=r$.
\end{thm}
\noindent
Since the inclusion $S^m_{M,\delta}\subset B^{r,M}_{\infty,\infty}S^m_{M,\delta}$ is true for all $r>0$, a straightforward consequence of Theorem \ref{sobolev-holder-cont} is the following
\begin{cor}\label{smooth-cont}
If $a\in S^m_{M,\delta}$, for $m\in\mathbb{R}$ and $\delta\in[0,1[$, then \eqref{boundedness2} is true for all $s\in\mathbb{R}$. If $\delta=1$, \eqref{boundedness2} is true for all $s>0$.
\end{cor}

\section{Microlocal properties}\label{mcl}
In this section we review some known microlocal tools and properties concerning the pseudodifferential operators introduced above. For the proofs of the results collected below, the reader is addressed to \cite{GM4}. In the sequel, we will set $T^\circ\mathbb R^n:=\mathbb R^n\times(\mathbb R^n\setminus\{0\})$, and $M=(m_1,\dots, m_n)$ will be a vector under the assumptions of \S 2.
\begin{defn}\label{Mcone}
We say that a set $\Gamma_M\subseteq\mathbb{R}^n\setminus\{0\}$ is $M-$conic, if
$$
\xi\in\Gamma_M\quad\Rightarrow\quad t^{1/M}\xi\in\Gamma_M\,\,,\,\,\forall\,t>0\,.
$$
\end{defn}
\begin{defn}\label{Melliptic}
A symbol $a\in S^m_{M,\delta}$ is microlocally $M-$elliptic at $(x_0,\xi_0)\in T^{\circ}\mathbb{R}^n$ if there exist an open neighborhood $U$ of $x_0$ and an $M-$conic open neighborhood $\Gamma_M$ of $\xi_0$ such that for $c_0>0$, $\rho_0>0$:
\begin{equation}\label{Mellipticity}
|a(x,\xi)|\ge c_0\langle\xi\rangle_M^m\,,\quad\forall\,(x,\xi)\in U\times\Gamma_M\,,\quad |\xi|_M>\rho_0\,.
\end{equation}
Moreover the characteristic set of $a\in S^{m}_{M,\delta}$ is ${\rm Char}(a)\subset T^{\circ}\mathbb{R}^n$ defined by
\begin{equation}
(x_0,\xi_0)\in T^{\circ}\mathbb{R}^n\setminus{\rm Char}(a)\,\,\Leftrightarrow\,\,\,\,a\,\, \text {is\,\,microlocally\,\,\emph{M-}elliptic\,\,at}\,\,(x_0,\xi_0)\,.
\end{equation}
\end{defn}
\begin{defn}\label{reg_symbol}
We say that $a\in \mathcal S'(\mathbb R^n)$ is {\it microlocally regularizing} on $U\times \Gamma_M$ if $a_{|\,\,U\times\Gamma_M}\in C^{\infty}(U\times\Gamma_M)$ and for every $m>0$ and all $\alpha,\beta\in\mathbb{Z}^n_+$ a positive constant $C_{m,\alpha,\beta}>0$ exists in such a way that:
\begin{equation}\label{microregineq}
|\partial^{\alpha}_{\xi}\partial^{\beta}_x a(x,\xi)|\le C_{m,\alpha,\beta}(1+|\xi|)^{-m}\,,\quad\forall\,(x,\xi)\in U\times\Gamma_M\,.
\end{equation}
\end{defn}
\begin{prop}{\bf (Microlocal parametrix).}\label{microparametrix}
Assume that $0\le\delta<1/m^*$. Then $a\in S^m_{M,\delta}$ is microlocally $M-$elliptic at $(x_0,\xi_0)\in T^{\circ}\mathbb{R}^n$ if and only if there exist symbols ${b},{c}\in S^{-m}_{M,\delta}$ such that
\begin{equation}\label{leftrightinverse}
a(x,D)b(x,D)=I+r(x,D)\qquad\mbox{and}\qquad c(x,D)a(x,D)=I+l(x,D)\,,
\end{equation}
being $I$ the identity operator and the symbols $r(x,\xi)$, $l(x,\xi)$ microlocally regularizing at $(x_0,\xi_0)$.
\end{prop}
\begin{defn}\label{microsobolev}
For $(x_0,\xi_0)\in T^{\circ}\mathbb{R}^n$, $s\in \mathbb R$, we define $mcl B^{s,M}_{\infty,\infty}(x_0,\xi_0)$ as the set of   $u\in\mathcal{S}^{\prime}(\mathbb{R}^n)$   such that:
\begin{equation}\label{microsobolevprop}
\psi(D)(\phi u)\in B^{s,M}_{\infty,\infty}\,,
\end{equation}
where $\phi\in C^{\infty}_0(\mathbb{R}^n)$ is identically one in a neighborhood of $x_0$, $\psi(\xi)\in S^0_{M}$ is a  symbol identically one on $\Gamma_M\cap\{|\xi|_M>\varepsilon_0\}$, for $0<\varepsilon_0<\vert \xi_0\vert_M$, and finally $\Gamma_M\subset\mathbb{R}^n\setminus\{0\}$ is an $M-$conic neighborhood  of $\xi_0$.\\
Under the same assumptions, we also write
$$
(x_0,\xi_0)\notin WF_{B^{s,M}_{\infty,\infty}}(u)\,.
$$
The set $WF_{B^{s, M}_{\infty,\infty}}(u)\subset T^{\circ}\mathbb{R}^n$ is called the $B^{s, M}_{\infty,\infty}-${\it wave front set} of $u$.
\newline
We say that a distribution satisfying the previous definition is {\it microlocally in $B^{s,M}_{\infty,\infty}$ at $(x_0,\xi_0)$}. Moreover the closed set $WF_{B^{s, M}_{\infty,\infty}}(u)$ is $M-conic$ in the $\xi$ variable.\\

Finally we say that   $x_0\notin B^{s, M}_{\infty,\infty}-{\rm singsupp}\,(u)$ if and only if there exists a function $\phi\in C^{\infty}_0(\mathbb{R}^n)$, $\phi\equiv 1$ in some open neighborhood of $x_0$, such that $\phi u\in B^{s, M}_{\infty,\infty}$.
\end{defn}
\begin{prop}\label{multiplication}
If $u\in mcl B^{s, M}_{\infty,\infty}(x_0,\xi_0)$, with $(x_0,\xi_0)\in T^{\circ}\mathbb R^n$, then for any $\varphi\in C^\infty_0(\mathbb R^n)$, such that $\varphi(x_0)\neq 0$, $\varphi u\in mcl B^{s, M}_{\infty,\infty} (x_0,\xi_0)$ .
\end{prop}
\begin{prop}\label{proiezione}
Let $\pi_1$ be the canonical projection of $T^{\circ}\mathbb{R}^n$ onto $\mathbb{R}^n$, $\pi_1(x,\xi)=x$. For every $u\in\mathcal{S}^{\prime}(\mathbb{R}^n)$ and $s\in\mathbb{R}$ we have:
$$
B^{s, M}_{\infty,\infty}-{\rm singsupp}(u)=\pi_1(WF_{B^{s, M}_{\infty,\infty}}(u))\,.
$$
\end{prop}
\begin{thm}\label{microsobolevaction}
Let $a\in S^m_{M,\delta}$ for $\delta\in[0,1/m^*[$, $m\in\mathbb R$ and $(x_0,\xi_0)\in T^{\circ}\mathbb{R}^n$. Then for all $s\in\mathbb{R}$
\begin{equation}\label{microcont}
u\in mcl B^{s+m, M}_{\infty, \infty}(x_0,\xi_0)\quad\Rightarrow\quad a(x,D)u\in mcl B^{s, M}_{\infty, \infty}(x_0,\xi_0)\,.
\end{equation}
\end{thm}
\begin{thm}\label{singularities}
Let $a\in S^m_{M,\delta}$, for $m\in\mathbb{R}$, $\delta\in[0,1/m^*[$, be microlocally $M-$elliptic at $(x_0,\xi_0)\in T^{\circ}\mathbb{R}^n$. For $s\in\mathbb{R}$ assume that $u\in\mathcal{S}^{\prime}(\mathbb{R}^n)$ fulfills $a(x,D)u\in mcl B^{s, M}_{\infty, \infty}(x_0,\xi_0)$. Then $u\in mcl B^{s+m, M}_{\infty, \infty}(x_0,\xi_0)$.
\end{thm}
\noindent
As a consequence of Theorems \ref{microsobolevaction}, \ref{singularities}, the following holds.
\begin{cor}\label{wavefront}
For $a\in S^m_{M,\delta}$, $m\in\mathbb{R}$, $\delta\in[0,1/m^*[$ and $u\in\mathcal{S}^{\prime}(\mathbb{R}^n)$, the inclusions
\begin{equation}\label{frontinclusions}
WF_{B^{s, M}_{\infty, \infty}}(a(x,D)u)\subset WF_{B^{s+m, M}_{\infty, \infty}}(u)\subset WF_{B^{s, M}_{\infty, \infty}}(a(x,D)u)\cup{\rm Char}(a)
\end{equation}
hold true for every $s\in\mathbb{R}$.
\end{cor}

\section{Non regular symbols}\label{appl}
In this section, the microlocal regularity results discussed in \S \ref{mcl} are applied to obtain microlocal regularity results for a linear partial differential equation of quasi-homogeneous order $m\in\mathbb{N}$ of the form
\begin{equation}\label{APPL:1}
A(x,D)u:=\sum_{\alpha\cdot 1/M \leq m}a_\alpha(x) D^\alpha u=f(x)\,,
\end{equation}
where $D^{\alpha}:=(-i)^{|\alpha|}\partial^{\alpha}$ and the coefficients $a_{\alpha}$ belong to the Besov space $B^{r,M}_{\infty,\infty}$ of positive order $r$. It is clear that $A(x,\xi)=\sum_{\alpha\cdot 1/M \leq m}a_\alpha(x) \xi^\alpha\in B^{r, M}_{\infty, \infty}S^m_M$.
\newline
We assume that $A(x,D)$ is {\it microlocally} $M-$elliptic at a given point $(x_0,\xi_0)\in T^{\circ}\mathbb{R}^n$; according to Definition \ref{Melliptic} and the quasi-homogeneity of the norm $\vert \xi\vert_M$, this means that there exist an open neighborhood $U$ of $x_0$ and an open $M-$conic neighborhood $\Gamma_M$ of $\xi_0$ such that the $M-${\it principal symbol} of $A(x,D)$ satisfies
\begin{equation}\label{principal}
A_m(x,\xi)=\sum_{\alpha\cdot 1/M=m}a_\alpha(x)\xi^\alpha\neq 0, \quad \text{for}\, (x,\xi)\in U\times \Gamma_M.
\end{equation}
The forcing term $f$ is assumed to be in some $B^{s, M}_{\infty, \infty}$, with  a suitable order of smoothness $s$, {\it microlocally} at $(x_0,\xi_0)$ (cf. Definition \ref{microsobolev}).

\begin{thm}\label{APPLICATION}
Let $A(x,D)u=f$ be a linear partial differential equation, as in \eqref{APPL:1}, with coefficients in  the space $B^{r,M}_{\infty,\infty}$ of positive order $r$. Assume that $A(x,D)$ is microlocally $M$-elliptic at $(x_0,\xi_0)\in T^{\circ}\mathbb{R}^n$. If $f\in mcl B^{s-m,M}_{\infty,\infty}(x_0,\xi_0)$ and $u\in B^{s-\delta r,M}_{\infty,\infty}$, for $0<\delta<1/m^*$ and $(\delta-1)r+m<s\le r+m$, then $u\in mcl B^{s,M}_{\infty,\infty}(x_0,\xi_0)$.
\end{thm}
\begin{rem}
{\rm Assuming in \eqref{APPL:1} $A(x,D)$ with coefficients in $B^{r,M}_{\infty,\infty}$, $r>0$, $u$ a priori in $B^{s-\delta r,M}_{\infty,\infty}$ for $(\delta-1)r+m<s\le r+m$, $\delta\in]0,1/m^*[$, we obtain
$$
\begin{array}{ll}
WF_{B^{s,M}_{\infty,\infty}}(u)\subset WF_{B^{s-m,M}_{\infty,\infty}}(A(x,D)u)\cup{\rm Char}(A)\,.
\end{array}
$$}
\end{rem}
\noindent
Following  \cite{Ta1}, \cite{GM3A}, non smooth symbols in $B^{r,M}_{\infty,\infty}S^m_M$ can be decomposed, for a given $\delta\in]0,1]$, into the sum of a smooth symbol in $S^m_{M,\delta}$ and a non smooth symbol of lower order. Namely,
let $\phi$ be a fixed $C^{\infty}$ function such that $\phi(\xi)=1$
for $\langle\xi\rangle_M\le 1$ and $\phi(\xi)=0$ for
$\langle\xi\rangle_M> 2$. For given $\varepsilon>0$ we set
 $\phi(\varepsilon^{1/M}\xi):=\phi(\varepsilon^{1/m_1}\xi_1,\dots,\varepsilon^{1/m_n}\xi_n)$.
\newline
Any symbol $a(x,\xi)\in B^{r,M}_{\infty,\infty}S^m_M$ can be split in
\begin{equation}\label{split}
a(x,\xi)= a^\#(x,\xi)+a^\natural(x,\xi),
\end{equation}
where  for some $\delta\in ]0,1]$
$$
a^{\#}(x,\xi):=\sum\limits_{h=-1}^{\infty}\phi(2^{- h\delta/M}D_x)a(x,\xi)\varphi_h(\xi).
$$
One can prove the following proposition (see \cite[Proposition 3.9]{GM3A} and \cite{Ta1} for the proof):
\begin{prop}\label{pro:2}
If $a(x,\xi)\in B^{r,M}_{\infty,\infty}S^{m}_M$, with $r>0$, $m\in\mathbb{R}$, and
$\delta\in]0,1]$, then $a^{\#}(x,\xi)\in S^m_{M,\delta}$ and $a^{\natural}(x,\xi)\in B^{r,M}_{\infty,\infty}S^{m-r\delta}_{M,\delta}$.
\end{prop}
\noindent
\begin{prop}\label{SHARPEL}
Assume that $a(x,\xi)\in B^{r,M}_{\infty,\infty}S^{m}_{M}$, $m\in\mathbb R$,  is microlocally $M-$elliptic at $(x_0,\xi_0)\in T^{\circ}\mathbb{R}^n$, then for any $\delta\in]0,1]$, $a^\#(x,\xi)\in S^{m}_{M,\delta}$ is still microlocally $M-$elliptic at $(x_0,\xi_0)$.
\end{prop}
\begin{proof}
The microlocal $M-$ellipticity of $a(x,\xi)$ yields the existence of positive constants $c_1$, $\rho_1$ such that
\begin{equation}\label{microellitticita}
|a(x,\xi)|\ge c_1\langle\xi\rangle^m_M\,,\,\,{\rm when}\,\,(x,\xi)\in U\times\Gamma_M\,\,{\rm and}\,\,|\xi|_M>\rho_1\,,
\end{equation}
where $U$ is a suitable open neighborhood of $x_0$ and $\Gamma_M$ an open $M-$conic neighborhood of $\xi_0$. On the other hand, for any $\rho_0>0$ we can find a positive integer $h_0$, which increases together with $\rho_0$, such that $\varphi_h(\xi)=0$ as long as $\vert \xi\vert_M>\rho_0$ and $h=-1,\dots, h_0-1$. We can then write:
\begin{equation}\label{SHARPEL:1}
a^\#(x,\xi)=\sum_{h=h_0}^\infty \phi\left( 2^{-h\delta/M}D_x\right)a(x,\xi)\varphi_h(\xi), \quad \vert \xi\vert_M>\rho_0\,.
\end{equation}
Set for brevity  $\phi\left( 2^{-h\delta/M}\cdot\right)=\phi_h(\cdot)$.
\newline
By means of \eqref{SHARPEL:1}, the Cauchy-Schwarz inequality and \cite[Lemma 3.8]{GM3A}, when $|\xi|_M>\rho_0$ we can estimate
\begin{equation*}
\begin{array}{l}
\vert a^\#(x,\xi)-a(x,\xi)\vert^2\\
=\left\vert\sum\limits_{h=h_0}^\infty \left(\phi_h(D_x) -I\right)a(x,\xi) \varphi_h(\xi)\right\vert^2\\
=\sum\limits_{h=h_0}^\infty \sum\limits_{k=h-N_0}^{h+N_0}\left\langle \left(\phi_h(D_x)) - I\right) a(x,\xi) \varphi_h(\xi) , \left(\phi_k(D_x) -I\right) a(x,\xi) \varphi_k(\xi) \right\rangle\\
=\sum\limits_{t=-N_0}^{N_0} \sum\limits_{h=h_0}^{\infty}\left\langle \left(\phi_h(D_x) -I\right) a(x,\xi) \varphi_h(\xi) , \left(\phi_{h+t}(D_x) -I\right) a(x,\xi) \varphi_{h+t}(\xi) \right\rangle\\
\leq\sum\limits_{t=-N_0}^{N_0}\!\sum\limits_{h=h_0}^\infty \!\Vert \left(\phi_h(D_x))- I\right)a(\cdot,\xi)\Vert_{L^\infty}\vert \varphi_h(\xi)\vert\\
\times\Vert \left(\phi_{h+t}(D_x)-I\right) a(\cdot,\xi)\Vert_{L^\infty}\vert \varphi_{h+t}(\xi)\vert\\
\leq C^2 \sum\limits_{t=-N_0}^{N_0} \sum\limits_{h=h_0}^{\infty}2^{-h\delta r} 2^{-(h+t)\delta r}\Vert a(\cdot,\xi) \Vert^2_{B^{r,M}_{\infty,\infty}}\\
\leq C^2\sum\limits_{h=h_0}^\infty 2^{-2h\delta r}\Vert a(\cdot,\xi) \Vert^2_{B^{r,M}_{\infty,\infty}}\leq
C^2 2^{-2h_0\delta r}\Vert a(\cdot,\xi) \Vert^2_{B^{r,M}_{\infty,\infty}},
\end{array}
\end{equation*}
where  $C$ denotes different positive constants depending only on $\delta, N_0$ and $r$. Since $\Vert a(\cdot,\xi)\Vert_{B^{r,M}_{\infty,\infty}}\le c^*\langle\xi\rangle^m_M$, let us fix $\rho_0$ large enough to have $C2^{-h_0\delta r}< \frac{c_1}{2c^*}$ (with $c_1$ from \eqref{microellitticita}). Then for $(x,\xi)\in U\times\Gamma_M$ and $|\xi|_M>{\rm max}\,\{\rho_0,\rho_1\}$
\begin{equation}\label{SHARPEL:3}
\begin{array}{l}
\vert a^\#(x,\xi)\vert\ge \vert a(x,\xi) \vert-\vert a^\#(x,\xi)-a(x,\xi)\vert\ge\frac{c_1}{2}\langle\xi\rangle^m_M
\end{array}
\end{equation}
follows and the proof is concluded.
\end{proof}
\noindent
{\it Proof of Theorem \ref{APPLICATION}}

\vspace{.1cm}
\noindent
Consider now the linear partial differential equation \eqref{APPL:1}, with $A(x,D)$ microlocally $M-$elliptic at $(x_0,\xi_0)$. For an arbitrarily fixed $\delta\in]0,1/m^*[$, we split the symbol $A(x,\xi)$ as $A(x,\xi)=A^{\#}(x,\xi)+A^{\natural}(x,\xi)$, according to Proposition \ref{pro:2}. In view of Propositions \ref{SHARPEL}, \ref{microparametrix} there exists a smooth symbol $B(x,\xi)\in S^{-m}_{M, \delta}$ such that
$$
B(x,D)A^{\#}(x,D)=I+R(x,D)\,,
$$
where $R(x,D)$ is microlocally regularizing at $(x_0,\xi_0)$.
\newline
Applying now $B(x,D)$ to both sides of \eqref{APPL:1}, on the left, we obtain:
\begin{equation}\label{APPL:4}
u=B(x,D)f-R(x,D)u-B(x,D)A^{\natural}(x,D)u\,.
\end{equation}
Assume that $f\in mcl B^{s-m, M}_{\infty, \infty}(x_0,\xi_0)$ and $u\in B^{s-\delta r, M}_{\infty, \infty}$ for $(\delta-1)r+m<s\le r+m$. Since $A^{\natural}(x,\xi)\in B^{r,M}_{\infty,\infty}S^{m-r\delta}_{M,\delta}$, one can apply Theorem \ref{sobolev-holder-cont} and Corollary \ref{smooth-cont} to find that $B(x,D)A^{\natural}(x,D)u\in B^{s, M}_{\infty, \infty}$; moreover Theorem \ref{microsobolevaction} and Corollary \ref{smooth-cont} give $B(x,D)f\in mcl B^{s, M}_{\infty, \infty}(x_0,\xi_0)$ and $R(x,D)u\in B^{s, M}_{\infty, \infty}$. This shows the result of Theorem \ref{APPLICATION}.

\vspace{.1cm}
\noindent
By means of the argument  stated above, we obtain the following general result for non regular pseudodifferential operators.
\begin{cor}\label{APPLICATION2}
For $a(x,\xi)\in B^{r,M}_{\infty,\infty}S^m_M$, $r>0$, $u$ belonging a priori to $B^{s-\delta r,M}_{\infty,\infty}$, for $(\delta-1)r+m<s\le r+m$, $\delta\in]0,1/m^*[$, we have
$$
\begin{array}{ll}
WF_{B^{s,M}_{\infty,\infty}}(u)\subset WF_{B^{s-m,M}_{\infty,\infty}}(a(x,D)u)\cup{\rm Char}(a)\,.
\end{array}
$$
\end{cor}

\section{Some applications to non linear equations}\label{nl}
In this section, we apply the previous results to the study of microlocal properties for a class of quasi-linear and fully non linear partial differential equation of weighted elliptic type.
\newline
For $M=(m_1,\dots, m_n)$, satisfying the assumptions in \S 2, and a given positive integer $m$, let us first consider the quasi-linear equation of quasi-homogeneous type
\begin{equation}\label{qleq}
\sum\limits_{\alpha\cdot 1/M\le m}a_\alpha(x, D^{\beta}u)_{\beta\cdot 1/M\le m-1}D^\alpha u=f(x)\,,
\end{equation}
where $a_\alpha(x,\zeta)\in C^\infty(\mathbb R^n\times\mathbb C^N)$ are given functions of the vectors $x\in\mathbb R^n$, $\zeta=(\zeta_\beta)_{\beta\cdot1/M\le m-1}\in\mathbb C^N$ and $f(x)$ is a given forcing term. We assume that the equation \eqref{qleq} is microlocally $M-$elliptic at a given point $(x_0,\xi_0)\in T^\circ\mathbb R^n$, meaning that the $M-$principal symbol $A_m(x,\xi,\zeta):=\sum\limits_{\alpha\cdot 1/M\le m}a_\alpha(x, \zeta)\xi^\alpha$ of the differential operator in the left-hand side of the equation satisfies
\begin{equation}\label{qom}
A_m(x,\xi,\zeta)\neq 0\,\,\,\mbox{for}\,\,(x,\xi)\in U\times\Gamma_M,
\end{equation}
where $U$ is a suitable neighborhood of $x_0$ and $\Gamma_M$ a suitable $M-$conic neighborhood of $\xi_0$.
\newline
Under the previous assumptions, we may prove the following
\begin{thm}\label{qlmicroreg}
Consider  $r>0$, $0<\delta<\frac1{m^*}$, $\sigma<s\le r+m$, where $\sigma=\sigma_{r,\delta,m}:=\max\{(\delta-1)r+m, r+m-1\}$. Let $u\in B^{r+m-1, M}_{\infty, \infty}\cap B^{s-\delta r, M}_{\infty, \infty}$ be a solution to the equation \eqref{qleq}, microlocally $M-$elliptic at $(x_0,\xi_0)\in T^{\circ}\mathbb{R}^n$, with $f\in mcl B^{s-m, M}_{\infty, \infty}(x_0, \xi_0)$. Then $u\in mcl B^{s, M}_{\infty, \infty}(x_0,\xi_0)$.
\end{thm}
\begin{proof}
In view of Theorems \ref{sobolev-holder-cont} and \ref{COMP}, from $u\in B^{r+m-1, M}_{\infty, \infty}$ it follows that $D^\beta u\in B^{r, M}_{\infty, \infty}$, as long as $\beta\cdot 1/M\le m-1$, hence $a_\alpha(\cdot, D^{\beta}u)_{\beta\cdot 1/M}\in B^{r, M}_{\infty, \infty}$.
\newline
Then, since $u\in B^{s-\delta r, M}_{\infty, \infty}$, $0<\delta<\frac1{m^*}$ and $(\delta-1)r+m<s\le r+m$, we can apply Theorem \ref{APPLICATION} to $A(x,\xi):=\sum\limits_{\alpha\cdot 1/M\le m}a_{\alpha}(x,D^\beta u)_{\beta\cdot 1/M\le m-1}\xi^\alpha\in B^{r, M}_{\infty, \infty}S^m_M$, which is microlocally $M-$elliptic at $(x_0,\xi_0)$ because of \eqref{qom}. This shows the result.
\end{proof}
\noindent
We observe that if $r\delta\ge 1$, then $B^{r+m-1, M}_{\infty, \infty}\cap B^{s-\delta r, M}_{\infty, \infty}=B^{r+m-1, M}_{\infty, \infty}$, since $s-\delta r\le r+m-\delta r\le r+m-1$. If $r>m^*$, we may always find $\delta^*\in]0,1/m^*[$ such that $r\delta^*\ge 1$, the minimum admissible value being $\delta^*=\frac1{r}$. Then the microregularity result of Theorem \ref{qlmicroreg} applies to an arbitrary solution $u\in B^{r+m-1, M}_{\infty, \infty}$ of the equation \eqref{qleq} with $\delta^*=\frac1{r}$ (note that $\sigma=r+m-1$ when $r>m^*$). We can then state the following
\begin{cor}
For $r>m^*$, $r+m-1<s\le r+m$, let $u\in B^{r+m-1, M}_{\infty, \infty}$ be a solution to the equation \eqref{qleq}, microlocally $M-$elliptic at $(x_0,\xi_0)\in T^{\circ}\mathbb{R}^n$, with $f\in mcl B^{s-m, M}_{\infty, \infty}(x_0, \xi_0)$. Then $u\in mcl B^{s, M}_{\infty, \infty}(x_0,\xi_0)$.
\end{cor}
\noindent
Let us consider now the fully non linear equation
\begin{equation}\label{nleq}
F(x,D^\alpha u)_{\alpha\cdot 1/M\le m}=f(x)\,,
\end{equation}
where $m$ is a given positive integer, $F(x,\zeta)\in C^\infty(\mathbb R^n\times\mathbb C^N)$ is a known function of $x\in\mathbb R^n$, $\zeta=(\zeta_\beta)_{\beta\cdot 1/M\le m-1}\in \mathbb C^N$.
\newline
Let the equation \eqref{nleq} be microlocally $M-$elliptic at $(x_0,\xi_0)\in T^\circ\mathbb R^n$, meaning that the {\it linearized} $M-$principal symbol $A_m(x,\xi,\zeta):=\sum\limits_{\alpha\cdot 1/M=m}\frac{\partial F}{\partial\zeta_\alpha}(x,\zeta)\xi^\alpha$ satisfies
\begin{equation}\label{nlell}
\sum\limits_{\alpha\cdot 1/M=m}\frac{\partial F}{\partial\zeta_\alpha}(x,\zeta)\xi^\alpha\neq 0\,\,\,\mbox{for}\,\,(x,\xi)\in U\times\Gamma_M,
\end{equation}
for $U$ a suitable neighborhood of $x_0$ and $\Gamma_M$ a suitable $M-$conic neighborhood of $\xi_0$. Under the assumptions above, we may prove the following
\begin{thm}\label{nlmicroreg}
For $r>0$, $0<\delta<\frac1{m^*}$, assume that $u\in B^{r+m, M}_{\infty, \infty}$, satisfying in addition
\begin{equation}\label{der}
\partial_{x_j}u\in B^{r+m-\delta r, M}_{\infty, \infty}\,,\quad j=1,\dots,n\,,
\end{equation}
is a solution to the equation \eqref{nleq}, microlocally $M-$elliptic at $(x_0,\xi_0)\in T^{\circ}\mathbb{R}^n$. If moreover the forcing term satisfies
\begin{equation}\label{der_f}
\partial_{x_j}f\in mcl B^{r, M}_{\infty, \infty}(x_0,\xi_0)\,,\quad j=1,\dots,n\,,
\end{equation}
we obtain 
\begin{equation}\label{der_mcl}
\partial_{x_j}u\in mcl B^{r+m, M}_{\infty, \infty}(x_0,\xi_0)\,,\quad j=1,\dots, n\,.
\end{equation}
\end{thm}
\begin{proof}
For each $j=1,\dots,n$, we differentiate \eqref{nleq} with respect to $x_j$  finding that $\partial_{x_j}u$ must solve the linearized equation
\begin{equation}\label{lineq}
\sum\limits_{\alpha\cdot 1/M\le m}\frac{\partial F}{\partial\zeta_\alpha}(x,D^\beta u)_{\beta\cdot 1/M\le m}D^\alpha\partial_{x_j}u=\partial_{x_j}f-\frac{\partial F}{\partial x_j}(x, D^\beta u)_{\beta\cdot 1/M\le m}\,.
\end{equation}
From Theorems \ref{sobolev-holder-cont} and \ref{COMP}, $u\in B^{r+m, M}_{\infty, \infty}$ yields that $\frac{\partial F}{\partial\zeta_\alpha}(\cdot ,D^\beta u)_{\beta\cdot 1/M\le m}\in B^{r, M}_{\infty, \infty}$. Because of the hypotheses \eqref{der}, \eqref{der_f}, for each $j=1,\dots, n$, Theorem \ref{APPLICATION} applies to $\partial_{x_j}u$, as a solution of the equation \eqref{lineq} (which is microlocally $M-$elliptic at $(x_0,\xi_0)$ in view of \eqref{nlell}), taking $s=r+m$. This proves the result.
\newline
\end{proof}
\begin{lem}
For every $s\in\mathbb R$, assume that $u, \partial_{x_j}u\in B^{s, M}_{\infty, \infty}$ for all $j=1,\dots,n$. Then $u\in B^{s+1/m^*, M}_{\infty, \infty}$. The same is still true if the Besov spaces $B^{s, M}_{\infty, \infty}, B^{s+1/m^*, M}_{\infty, \infty}$ are replaced by  $mcl B^{s, M}_{\infty, \infty}(x_0,\xi_0), mcl B^{s+1/m^*, M}_{\infty, \infty}(x_0, \xi_0)$ at a given point $(x_0,\xi_0)\in T^\circ\mathbb R^n$.
\end{lem}
\begin{proof}
Let us argue for simplicity in the case of the spaces $B^{s, M}_{\infty, \infty}$, the microlocal case being completely analogous.
\newline
In view of Theorem \ref{sobolev-holder-cont}, that $u$ belongs to $B^{s+1/m^*, M}_{\infty, \infty}$ is completely equivalent to show that $\langle D\rangle_M^{1/m^*}u\in B^{s, M}_{\infty, \infty}$. By the use of the known properties of the Fourier transform, we may rewrite $\langle D\rangle_M^{1/m^*}u$ in the form
\begin{equation*}
\langle D\rangle_M^{1/m^*}u=\langle D\rangle_M^{1/m^*-2}u+\sum\limits_{j=1}^n\Lambda_{j, M}(D)(D_{x_j}u)\,,
\end{equation*}
where $\Lambda_{j,M}(D)$ is the pseudodifferential operator with symbol $\langle\xi\rangle_M^{1/m^*-2}\xi_j^{2m_j-1}$, that is
\begin{equation*}
\Lambda_{j,M}(D)v:=\mathcal F^{-1}\left(\langle\xi\rangle_M^{1/m^*-2}\xi_j^{2m_j-1}\widehat v\right)\,,\quad j=1,\dots,n\,.
\end{equation*}
Since $\langle\xi\rangle_M^{1/m^*-2}\xi_j^{2m_j-1}\in S^{1/m^*-1/m_j}_M$, the result follows at once from Corollary \ref{smooth-cont}.
\end{proof}
\noindent
As a straightforward application of the previous lemma, the following consequence of Theorem \ref{nlmicroreg} can be proved.
\begin{cor}
Under the same assumptions of Theorem \ref{nlmicroreg} we have that $u\in mcl B^{r+m+\frac1{m^*}, M}_{\infty, \infty}(x_0,\xi_0)$.
\end{cor}
\begin{rem}
We notice that if $r\delta\ge 1$ then every function $u\in B^{r+m, M}_{\infty, \infty}$ automatically satisfies the condition \eqref{der}; indeed one can compute $\partial_{x_j}u\in B^{r+m-1/m_j, M}_{\infty, \infty}\subset B^{r+m-r\delta, M}_{\infty, \infty}$ being $1/m_j\le 1\le r\delta$ for each $j=1,\dots,n$. As already observed before, for $r>m^*$ we can always find $\delta^*\in]0,1/m^*[$ such that $r\delta^*\ge 1$ (it suffices to choose an arbitrary $\delta^*\in[1/r, 1/m^*[$); hence, applying Theorem \ref{nlmicroreg} with such a $\delta^*$ we conclude that if $r>m^*$ and the right-hand side $f$ of the equation \eqref{nleq} obeys the condition \eqref{der_f} at a point $(x_0,\xi_0)\in T^{\circ}\mathbb R^n$ then every solution $u\in B^{r+m, M}_{\infty, \infty}$ to such an equation satisfies the condition \eqref{der_mcl}; in particular $u\in mcl B^{r+m+1/m^*, M}_{\infty, \infty}(x_0,\xi_0)$.
\end{rem}

\end{document}